\documentclass[12pt]{article}

\usepackage{amssymb}
\usepackage{amsmath}
\usepackage{amscd}
\usepackage{amsthm}

\usepackage[dvips,final]{graphicx}
\usepackage[dvips]{geometry}
\usepackage{color} 
\usepackage{epsfig}
\usepackage{latexsym}
\usepackage{pstricks}

\newcommand{\C} {\mathbb C}

\newcommand{\h} {\hat}
\newcommand{\mc}{\mathcal}

\newtheorem{theorem}{Theorem}
\newtheorem{lemma}[theorem]{Lemma}

\newtheorem{proposition}[theorem]{Proposition}
\newtheorem*{definition}{Definition}
\newtheorem*{Example}{Example}

\title{On dynamical Teichm\"{u}ller spaces.}

\author{Carlos Cabrera  and Peter Makienko\\
        \\
        Instituto de Matem\'aticas,\\ Unidad Cuernavaca. UNAM}

\begin{document}

\maketitle
\footnotetext{This work was partially supported by PAPIIT project IN 100409.}
\begin{abstract}

\end{abstract}

Following ideas from a preprint of the second author, see \cite{Makaut}, we
investigate relations of dynamical Teichm\"{u}ller spaces with dynamical
objects. We also establish some connections with the theory of deformations of
inverse limits and laminations in holomorphic dynamics, see \cite{LM}.

\section{Introduction.}

Sullivan introduced the study of a dynamical Teichm\"{u}ller space, which we
denote by $T_1(R)$, associated to a rational function $R$. The space of orbits
of $T_1(R)$, under the action of an associated modular group $Mod_1(R)$,
coincides with the space $QC(R)$ of quasiconformal deformations of $R$. We
modify Sullivan's definition to get another Teichm\"{u}ller space $T_2(R)$, with
its corresponding modular group $Mod_2(R)$. In this situation, the
$J$-stability component is the space of orbits of $T_2(R)$, under the action of
$Mod_2(R)$. When $R$ is hyperbolic, the $J$-stability component is the
hyperbolic component of $R$.

There are natural inclusions of the space $T_1(R)$ into $T_2(R)$, and from the
group $Mod_1(R)$ into $Mod_2(R)$. We find that, properties of these
inclusions are related to the dynamics of $R$. With this at hand, we can
establish relations between the dynamics of $R$ and topological properties of
$T_2(R)$.

When the Julia set of $R$ is totally disconnected. The space $T_2(R)$ has a
laminated structure. In this way, we also realize $T_2(R)$ as the space of
deformations of the natural extension of $R$. The structure of the paper is as
follows.

In Section 2, we recall basic definitions and facts of the classical
Teichm\"{u}ller space $T_1(R)$.

In Section 3, we introduce $T_2(R)$ and show that, as in the case of $T_1(R)$,
is a complete metric space. In Theorem \ref{surjec}, we establish
characterizations for the path connectivity of $T_2(R)$. Using this, we prove
Theorem \ref{Julia.conn} stating that, when $R$ is a polynomial, the
connectivity of Julia set $J(R)$ is equivalent to the path connectivity of
$T_2(R)$.

In Section 4, we restrict to the case where $J(R)$ is homeomorphic to a Cantor
set. In this case, $T_2(R)$ is a trivial product of $T_1(R)$ times a totally
disconnected space. We finish the section giving a characterization of the
property that $J(R)$ is homeomorphic to a Cantor set in terms of properties of
$T_2(R)$.

Finally, in Section 5, we construct a realization of $T_2(R)$ as the space of
deformations of the natural extension of $R$.

\section{The Teichm\"{u}ller space $T_1(R)$.}

Given a rational map $R$, let us define the space
$T_1(R)=T(S_R)\times B_R$, where $S_R$ is a Riemann surface
associated to the Fatou set $F(R)$, $T(S_R)$ denotes the classical
Teichm\"{u}ller's space of $S_R$, and $B_R$ is the space of invariant Beltrami
differentials, defined on the Julia set $J(R)$, which are compatible with the
dynamics of $R$. That is, $B_R$ is the space of measurable $(1,-1)$
forms $\mu$ with $L^{\infty}$ norm bounded by $1$, satisfying the
conditions that $\mu$ is $0$ outside the Julia set $J(R)$ and
$R^*(\mu)=\mu$. For a more detailed account of the
definitions see \cite{McMAut} and \cite{McMSull}.

An equivalent way to define $T_1(R)$ is as the set of isotopy
classes of pairs $\langle[R_1],[\phi]\rangle$ where $\phi$ is a
quasiconformal conjugation of $R$ to the rational map $R_1 $. The
first modular group $Mod_1(R)$, is the group of all isotopy classes
of quasiconformal homeomorphisms of $\C$ commuting with $R$. The group
$Mod_1(R)$ acts on $T_1(R)$ with the action given by
$$[\phi]\langle [g],[\psi]\rangle = \langle [g], [\psi \circ
\phi^{-1}]\rangle.$$ A theorem due to Sullivan and McMullen states
that $Mod_1(R)$ acts on $T_1(R)$ as a group of isometries, for more
details see \cite{McMAut} and \cite{McMSull}. The formula
$T_1(R)/Mod_1(R)=QC(R)$, where $QC(R)$ is the space of
quasiconformal deformations of $R$, will play an important role in
what follows.

\section{The space $T_2(R)$.}

We will define a Teichm\"{u}ller space that generalizes the formula
$$T_1(R)/Mod_1(R)=QC(R)$$ for the $J$-stability component of $R$.

\begin{definition} Let $(X,d_1)$ and $(Y,d_2)$ be metric spaces, a map
$\phi:X\rightarrow Y$ is called $K$-quasiconformal, in Pesin's sense if, for
every
$x_0\in X$
$$\limsup_{r\rightarrow 0}\left\{
\frac{\sup\{|\phi(x_0)-\phi(x_1)|:|x_0-x_1|<r\}}{\inf\{
|\phi(x_0)-\phi(x_1)|:|x_0-x_1|<r\}}\right\}\leq K.$$
\end{definition}

Let us recall that two rational maps $R_1$ and $R_2$ are
$J$\textit{-equivalent}, if there is a homeomorphism
$h:J(R_1)\rightarrow J(R_2)$, which is quasiconformal in Pesin's
sense and conjugates $R_1$ to $R_2$.

Given a family of maps $\{R_w\}$ depending holomorphically on a
parameter $w\in W$, a map $R_{w_0}$ in $\{R_w\}$ is called
$J$\textit{-stable} if, there is a neighborhood $V$ of $w_0$ such
that, $R_w$ is $J$-equivalent to $R_{w_0}$ for all $w\in V$, and the
conjugating homeomorphisms depend holomorphically on $w$.

We denote by $QC_J(R)$, the $J$-\textit{stability component} of a
rational map $R$. This is the path connected component of the
$J$-equivalence class of $R$ containing $R$. In \cite{MSS},
Ma\~{n}e, Sad and Sullivan proved that for every holomorphic family
of rational maps, the union of the $J$-stability components is open
and dense. When $R$ is hyperbolic, an application of the
$\lambda$-Lemma, for holomorphic motions around $J(R)$, shows that
$QC_J(R)$ coincides with $Hyp(R)$, the hyperbolic component of
$R$.

Let $R$ be a rational map, we define the space $X_n(R)$ as the set
of pairs $(h,U)$, where $U$ is an open
neighborhood of the Julia set $J(R)$, and
$h:U\rightarrow \C$ is a quasiconformal embedding such that
$$h\circ R\circ h^{-1}=R_h$$ is the restriction of rational map, with $\deg
R\leq \deg R_h \leq n$, wherever the conjugacy is well defined.

We say that $(h_1, U_1) \sim (h_2,U_2)$ in $X_n(R)$ if, and only if, there
exists open sets $V_1$ and $V_2$, satisfying $V_i\subset U_1\cap U_2$,
$J(R)\subset V_i$, for $i=1,2$, and a M\"{o}bius transformation
$\gamma:\C\rightarrow \C$ such that the following diagram commutes
$$\begin{CD} V_1 @>h_1>> \C\\ @VFVV @VV\gamma V\\ V_2 @>h_2>> \C
\end{CD}$$ and so that, $F$ is a map homotopic to the identity, with
a homotopy that commutes with $R$. 

With this equivalence relation on $X_n(R)$, we can take representatives $(h,U)$
such that, $U$ has nice dynamical properties. For instance, if $R$ is hyperbolic
we can always choose $U$ satisfying $R^{-1}(U)\subset U$.

Following classical Teichm\"{u}ller theory, the map $\gamma$ would
be a holomorphic map. However next proposition, which is actually a
folklore fact, justifies our definition.

\begin{proposition} Let $R_1$ and $R_2$ be rational maps, and $\gamma$ a
conformal map that conjugates $R_1$ to $R_2$ around a neighborhood
of $J(R_1)$. Then $\gamma$ is the restriction of a M\"{o}bius
transformation.
\end{proposition}

\begin{proof}
Let $U$ be the neighborhood around $J(R_1)$ on which $\gamma$ is
defined, and let $x$ be a point in $U$, we define
$\gamma_x(R_1(x))=R_2(\gamma(x))$ and analytically continue $\gamma$ on
$R_1(U)$ through arcs starting at $R_1(x)$. In this way, we obtain a, possibly
multivalued, extension $\gamma_x$ of $\gamma$. By construction, $\gamma_x$ also
conjugates $R_1$ to $R_2$. Now let $y\in R^{-1}_1(R_1(x))$, using the branch
induced by $y$, we can define another extension $\gamma_y$ of $\gamma$
putting $\gamma_y(R_1(x))=R_2(\gamma(x))$, and analytically continue $\gamma_y$
along paths. Now, $\gamma_x$ and $\gamma_y$ coincide in $U$, hence by the
Monodromy Theorem $\gamma_x=\gamma_y$ in all $R_1(U)$. Thus the extension of
$\gamma$ on $R_1(U)$ does not depend on branches and is a well defined
holomorphic map. By induction, we extend $\gamma$ to $\bigcup_m^\infty
R^m_1(U)$.
But, since $U$ contains $J(R_1)$, the set $\bigcup_m^\infty R^m_1(U)$ covers the
whole Riemann sphere, with exception of at most two points.  Hence, $\gamma$
extends to a unimodal holomorphic function defined on the sphere, so $\gamma$ is
a M\"{o}bius transformation. \end{proof}

Let $T_{2,n}(R)=X_n(R)/\sim$, this definition generalizes the notion
of the Teichm\"{u}ller space for a rational function. The space
$X_n(R)$ is extremely big, note that we can change the neighborhood $U$,
arbitrarily in the pair $(h,U)$, and still get the same point in
$T_{2,n}(R)$. For instance, the restriction of $h$ on a smaller neighborhood.
Consider the space $T_{2,n}(z^2)$ with $n\geq 3$,
this space contains all maps of the form $z^2+\lambda z^3$ for $\lambda$
small enough. In this paper, we will restrict to the case where
$n=\deg(R)$ and, in this situation, we will omit the subindex $n$.

Two quasiconformal maps $f:U\rightarrow V$ and $g:U'\rightarrow V'$, defined on
neighborhoods of $J(R)$, are equivalent $f\sim g$, if there exist $W\subset U
\cap U'$ on which $f$ and $g$ are homotopic, with a homotopy that commutes with
$R.$ We can define a modular group $$Mod_2(R)=\{h:U\rightarrow V \textnormal{
q.c}: h \textnormal{ commutes with } R, \, J(R)\subset U \}/\sim.$$

Let $R$ be a hyperbolic rational map, one can check that
$$Hyp(R)\cong T_{2}(R)/Mod_2(R).$$ Note that the group $Mod_2(R)$ does not
depend on $n$. For $n>\deg R$, the quotient $T_{2,n}(R)/Mod_2(R)$
forms a much bigger space containing $Hyp(R)$, it also contains
other components, coming from higher degrees, arranged on the
boundary of $Hyp(R)$. This construction allow us to consider, as basic
points of the Teichm\"{u}ller space, points that ``belong'' to the
boundary of other $T_2(R')$. For instance, $z^2$ ``belongs'' to the
boundary of the space $T_2(\lambda z^3 +z^2)$ for $\lambda$ close to
zero, but not zero. In fact, the same is true for $T_1(\lambda z^3
+z^2)$. Nevertheless, the complete picture is yet to be
understood.

Now, let us define a third modular group $Mod_3(R)$, as the group of maps
$\phi:J(R)\rightarrow J(R)$ which are quasiconformal in Pesin's sense and
commute with $R.$

One would be inclined to introduce a third Teichm\"{u}ller space $T_3(R)$. A
sensible definition for this space, is to consider the set of quasiconformal
maps, in the sense of Pesin, defined just in the $J(R)$ and commuting
with $R$. However, it is not clear how to relate this Teichm\"{u}ller space with
the usual quasiconformal theory. In other words, in general, is not clear if 
the natural map from $Mod_2(R)$ to $Mod_3(R)$ is surjective.  We can carry on
this
discussion when  the map $R$ is hyperbolic and, more generally, when the Julia
set is described as limits of telescopes with bounded geometry. In these
cases, every quasiconformal map, defined on the  Julia set and inducing an
isomorphism on telescopes, can be extended to a quasiconformal map defined on a
neighborhood of $J(R)$. For definition on telescopes see \cite{Sullconf}.

\subsection{The space $T_2(R)$ is a complete metric space.}
Consider the formula $$d([f],[g])=\inf \log K(g^{-1}\circ f),$$ where $K$
denotes the distortion, and the infimum is taken over all representatives of
the maps $f$ and $g$.
This formula defines a pseudodistance on equivalence classes of quasiconformal
maps. In particular, defines a distance on the space $T_1(R)$, see
\cite{McMSull}.

\begin{theorem}
The Teichm\"{u}ller pseudodistance on $T_2(R)$ defines a distance and, with
this distance, $T_2(R)$ is a complete metric space.
\end{theorem}

\begin{proof}

The map $d$ clearly is positive, reflexive and satisfies the triangle
inequality. Let us check that $d$ is non degenerate.

Let $(\phi_n,U_n)$ be a sequence of representative points in $T_2(R)$,
such that the distortion $K(\phi_n)$ converges to $1$. Note that the
neighborhoods $U_n$ may converge to the Julia set in the sense of Hausdorff.
Hence, let us check that the maps $\phi_n$ are eventually well defined over
a neighborhood $U$ of $J(R)$. Then show that, in $U$, the maps
$\phi_n$ converge to a holomorphic map $\phi$. This will finish the proof,
because if $d([f],[g])=0$, then $f$ and $g$ are related by a holomorphic map.

First let us assume that $R$ is hyperbolic. Consider a repelling fixed point
$x_0$ of $R$ in $J(R)$, and a neighborhood $W$ around $x_0$. Choose $W$
so that, the diameter $diam(W)$ is less than half the distance of $x_0$ to the
critical set of $R$. With this choice the map $R$ is injective in $W$. We
extend the definition of $\phi_n$ to $W$ using the formula
$\phi_n(R(z))=R(\phi_n(z))$. The same construction works around all repelling
periodic points. Since the map is hyperbolic, this construction extends the
definition of $\phi_n$ to a neighborhood $U$ of $J(R)$, that only depends on the
distance of $J(R)$ to the critical set. The space of quasiconformal maps with
bounded distortion is compact, then the maps $\phi_n$ converge to a holomorphic
map on $W$.

When $R$ is not hyperbolic, the argument is more subtle. Since there are
critical points on the boundary and nearby, the diameters of the corresponding
neighborhoods converge to zero. However, we still can extend the domains of 
$\phi_n$. To do so, take neighborhoods around the critical values in the Julia
set, and extend to the critical points using the formula
$\phi_n(R(z))=R(\phi_n(z))$.

A slight modification in the argument above also shows that every Cauchy
sequence in $T_2(R)$ converges, thus $T_2(R)$ is a complete metric
space. \end{proof}

\subsection{The homomorphisms $\alpha$ and $\beta$.}

Each class of maps in $Mod_1(R)$ belongs to a class of maps in $Mod_2(R)$,
and correspondingly in $Mod_3(R)$. So, we have the following chain of
homomorphisms
$$Mod_1(R)\overset{\alpha}
{\longrightarrow}Mod_2(R)\overset{\beta}{\longrightarrow} Mod_3(R).$$

The whole sphere is a neighborhood of the Julia set, hence a class of maps in
$T_1(R)$ uniquely determines a class of maps in $T_2(R)$. This gives a
map $H:T_1(R)\rightarrow T_2(R)$.
Let us remark that the map $H$, in general, is not injective nor surjective.
However, the properties of the map $H$ are connected with the
homomorphism $\alpha.$

\begin{proposition}\label{ker.a} For any rational map $R$, we have
$$H(T_1(R))\cong T_1(R)/\ker \alpha.$$
\end{proposition}

\begin{proof}
Consider the following commutative diagram
\begin{equation*}\tag{*}
 \begin{CD} T_1(R) @>H>> T_2(R)\\ @VVV @VV
V\\ QC(R) @>>> Hyp(R) \end{CD}
\end{equation*}
where the map, from $QC(R)$ to $Hyp(R)$, is an embedding with dense image. We
use the formulae $T_1(R)/Mod_1(R)=QC(R)$ and $T_2(R)/Mod_2(R)=Hyp(R)$.
Assume that $H(\phi_1)=H(\phi_2)$, then there are neighborhoods $U_1$,
$U_2$, $V_1$, $V_2$  and a M\"{o}bius map $\gamma$ such that the following
diagram commutes $$\begin{CD} U_1
@>\phi_1>> V_1\\ @VFVV @VV\gamma V\\ U_2 @>\phi_2>> V_2 \end{CD}$$ and, the map
$F=\phi_2^{-1}\circ \gamma\circ\phi_1$ is homotopic to $Id$ in $U$, with a
homotopy that commutes with dynamics. If $\phi_1\neq \phi_2$ in $T_2(R)$, then
the homotopy can not be extended to a global map in the plane. Since
$H(\phi_1)=H(\phi_2)$, the images of $\phi_1$ and $\phi_2$, under $H$, project
to the same element in $Hyp(R)$. By the commutativity of the diagram (*),
$\phi_1$ and $\phi_2$ project to the same element in $QC(R)$. Hence,
$\phi_1$ and $\phi_2$ are related by a non-trivial element $\psi\in Mod_1(R)$
satisfying $\alpha(\psi)=Id$. So we have $H(T_1(R))\cong T_1(R)/\ker \alpha.$
\end{proof}

\begin{theorem}\label{surjec} The following conditions are equivalent:
\begin{itemize}
\item The homomorphism $\alpha$ is surjective.
\item The set $H(T_1(R))$ is dense in $T_2(R)$.
\item The space $T_2(R)$ is path connected.\end{itemize}
\end{theorem}

\begin{proof}
Assume that the homomorphism $\alpha$ is surjective. Again we make use of the
diagram
(*). Given any $\epsilon>0$ and a point $x\in T_2(R)$, there exist $y\in
H(T_1(R))$
and $\phi\in Mod_2(R)$ such that $d(\phi(y),x)<\epsilon$. Since $\alpha$
is surjective, there exist $\psi \in Mod_1(R)$ such that $\alpha(\psi)=\phi$.
But this implies that $\phi(y)\in H(T_1(R))$. Thus the set $H(T_1(R))$ is dense
in
$T_2(R)$.

Let us assume that $H(T_1(R))$ is dense in $T_2(R)$, take two points $x$ and $y$
in $T_2(R)$, then there are two sequences $\{x_n\}$ and $\{y_n\}$ in $H(T_1(R))$
converging to $x$ and $y$, respectively. Since $T_1(R)$ is path connected, there
is a sequence of paths $\gamma_n$ in $H(T_1(R))$, with $\gamma_n(0)=x_n$ and
$\gamma_n(1)=y_n$. By analytical continuation along $\gamma_n$, we can
force the sequence $\{\gamma_n\}$ to converge uniformly to a path $\gamma$, in
$T_2(R)$, connecting $x$ with $y$. Hence $T_2(R)$ is path connected.

Let $\phi\in Mod_2(R)$, if $\phi(H(T_1(R)))\cap H(T_1(R))\neq \emptyset$, then
$\phi\in \alpha(Mod_1(R))$, and $\phi(H(T_1(R)))=H(T_1(R)).$
In the other hand, if $\phi\in Mod_2(R)\setminus \alpha(Mod_1(R))$, then
$\phi(H(T_1(R)))\cap H(T_1(R))= \emptyset$.
This shows that $T_2(R)$ is not path connected if $Mod_2(R)\setminus
\alpha(Mod_1(R))\neq \emptyset.$ In fact, $T_2(R)$ is decomposed into path
connected components by $H(T_1(R))$ and its orbit
under the action of $Mod_2(R)/\alpha(Mod_1(R)).$

\end{proof}

\begin{Example}\label{Blashke.sur} Let us consider the map $F(z)=z^n$, the Julia
set is the unit circle $\mathbb{S}^1$. Let $\phi\in Mod_2(F)$, by composing with
a rotation, we can assume that $\phi(1)=1$. Any orientation
preserving automorphism of the unit circle that fixes $1$, and commutes with the
dynamics of $F$, must be the identity. This is so, since such automorphism must
fix every point in the grand orbit of $1$, and every grand orbit is dense in
$\mathbb{S}^1$.
Thus, $\phi$ restricted to $\mathbb{S}^1$ is the identity. Taking a suitable
homotopic representative of $\phi$, we can assume that $\phi$ leaves a tubular
neighborhood of $S^1$ invariant. The dynamics on this tubular neighborhood have
a fundamental domain homeomorphic to an annulus. Thus $\phi$ induces a
quasiconformal automorphism of this annulus. The group of quasiconformal
automorphisms of an annulus is generated by a Dehn twist of angle $2\pi$.

Let $\tau$ be this generator. Since $\phi$ commutes with dynamics, $\tau$ most
be propagated to the grand orbit of the fundamental group. A preimage of $\tau$
has the angle $2\pi/n$. A forward image increases the angle by $2\pi n$. But
$\phi$ is defined on a neighborhood $U$ of $\mathbb{S}^1$. Then, $\tau$ only
iterates finitely many times in $U$. Thus, the total angle is bounded, and
then the mapd induced by $\tau$ in $U$ can be continuously deformed to the
identity. This extends to every map generated by $\tau$.

By the assumption above, any element in $Mod_2(F)$ is represented by a
rotation which can be globally extended to an element in $Mod_1(F)$. The
homomorphism $\alpha$ is surjective hence, by Theorem \ref{surjec}, $T_2(F)$
is path connected. If $G$ is a hyperbolic Blashke map, then $G$ restricts to a
degree $n$ expanding map on $\mathbb{S}^1$, so $G$ is locally conjugated to $F$.
If $G$ is a Blashke map, then the map $\alpha$ is surjective and $T_2(G)$ is
path connected.
\end{Example}

The previous example motivates the following proposition.

\begin{theorem}\label{Julia.conn} Let $P$ be a polynomial, then
$T_2(P)$ is path connected if, and only if, the Julia set $J(P)$ is connected.
\end{theorem}

\begin{proof}
Assume that $J(P)$ is not connected, then there exist at least two
disjoint Jordan curves $\gamma_1$ and $\gamma_2$, contained in the
Fatou set, such that $P(\gamma_1)=P(\gamma_2)$, and the interior of
each curve intersects a piece of the Julia set. We can take $\gamma_1$ and
$\gamma_2$ such that, the image of these curves do not intersect the
postcritical set. Let $\phi:=(\phi,U)$ be the element in $Mod_2(P)$, defined by
a Dehn twist on $\gamma_1$ and acting as the identity in $\gamma_2$. Using
dynamics, extend these actions to the grand orbit of $\gamma_1$ and $\gamma_2$.
Then, $\phi$ can
not be extended continuously to a global map in $Mod_1(P)$, commuting
with dynamics of $P$. This is because the action, of the extension
of $\phi$, is homotopically different in two preimages of $P(\gamma_1)$.

Now, let us suppose that $J(P)$ is connected and let $\phi$ be an
element in $Mod_2(P)$. We will extend $\phi$ to a globally
defined map in $Mod_1(P)$. Since $P$ is a polynomial, $\infty$ is a
superattracting fixed point. If $deg(P)=d$, by B\"{o}ttcher's Theorem, $P$ is
conjugated to $z^d$ on the basin of attraction $A_0(\infty)$.

As we showed in Example \ref{Blashke.sur}, $\phi$ can be extended to
$A_0(\infty)$ and, the action of $\phi$ on $A_0(\infty)$ is either a rotation
or the identity. But $J(P)=\partial(A_0(\infty))$, then the boundary of each
Fatou component is either fixed by $\phi$ or, is moved to another component by
a rotation. In either case, interchanges Fatou components univalently. Then, it
is enough to extend the map on each periodic component. Once it is done, we use
the
dynamics of $P$ to extend to preperiodic components.

Let us check that we can extend $\phi$ to every periodic Fatou
component $W$. There are three cases; if $W$ is hyperbolic then $P$
is conjugated on $W$ to a hyperbolic Blashke map so, by Example
\ref{Blashke.sur}, $\phi$ can be extended to $W$.

If $W$ is a Siegel disk, then $\phi$ is defined on $U$ a neighborhood of $J(P)$.
We can modify $\phi$ using a homotopy, so that $\phi$ leaves invariant a
rotational leaf $L$ of the Siegel foliation of $W$. Since $\phi$ is
quasiconformal in $U$, the restriction of $\phi$ to $L$ is quasi-regular. Hence,
we can radially extend $\phi$ to a quasiconformal map in $W$.

Finally, the case where $W$ is a parabolic Fatou component. Let
$K=W\setminus U$ be the compact set where the map $\phi$ is not defined. The
neighborhood $U$ contains a horodisk $D$, induced by the parabolic
dynamics of $P$, in $W$. It also contains all the $P^n$-preimages of $K$, for
a sufficiently large $n$. Thus $P^n$ has a lifting from $K$ to $U$.
Let $C_v=\{v_1,v_2,...,v_m\}$ be the set of critical values in $W$, and
$*$ be a given point in $(U\cap W)\setminus C_v$. By Hurwitz Theorem, the map
$P^n$ induces an isomorphism of the fundamental group $\pi_1(W\setminus C_v,*)$.

Hence, given a point $x$ in $K$ such that $P(x)\in D$. Take $y\in
P^{-n}(x)$, and define $\phi(x)=P^n(\phi(y))$. As a consequence of
the Hurwitz argument above, $\phi(x)$ does not depend on the point
$y$. Moreover, any homotopy that moves the point $y$, must move all
other elements in $P^{-n}(x)$. Since the map induced by $P^n$ in
$\pi_1(W\setminus C_v,*)$ is an isomorphism. Also, $P^{k}$ and $\phi$
are defined in $U$ and commute for all $k\geq n$. Thus we have
$P(\phi(x))=P^{n+1}(\phi(y)))=\phi(P^{n+1}(y))=\phi(P(x))$, so the
extension of $\phi$ in $P^{-1}(D)\cap K$ commutes with $P$. The extension is
quasiconformal since $P$ is holomorphic. Finally, using the dynamics
of $P$, we extend $\phi$ to $K$. \end{proof}

\section{Maps with totally disconnected Julia sets.}

We now restrict the discussion to the case where the Julia set $J(R)$ is
homeomorphic to a Cantor set. Under these conditions, we show that the
Teichm\"{u}ller space $T_2(R)$ has a product structure. We shortly remind the
proof of the following known fact.

\begin{lemma}\label{Dehn.lem} Let $P$ be a unimodal polynomial such that $J(P)$
is totally disconnected, then $Mod_1(P)$ is generated by a single Dehn twist.
\end{lemma}

\begin{figure}[htbp]
\begin{center}
\input{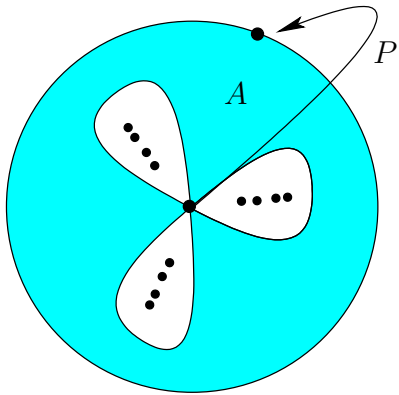}
\caption{Critical annulus for Cantor dynamics.}
\label{anillo.fig}
\end{center}
\end{figure}

\begin{proof} Let $d$ be the degree of $P$. Consider a simple close path
$\gamma$ through the critical value in the dynamical plane. The preimage of
$\gamma$ consist of $d$ closed loops, based on the critical point (see Figure
\ref{anillo.fig}). Let $A$ be the annulus defined by the intersection of
the interior of $\gamma$ with the exterior of $P^{-1}(\gamma)$. Any global
automorphism of $\C$, commuting with the dynamics of $P$, must leave the annulus
$A$ invariant. Hence, the group of such automorphisms is generated by a Dehn
twist defined on $A$.
\end{proof}

Let $S$ be a multiply connected Riemann surface with boundary such that, the
connected components of $\partial S$ are Jordan curves. The pure mapping class
group $Map(S)$ is defined by the set of topological automorphisms of
$S$, acting identically on the boundary, modulo a homotopic relation. This
homotopic relation is defined as follows, $f\sim g$ are equivalent if, and only
if, there
exist an isotopy $H$, from $f$ to $g$, such that $H|_{\partial S}=f|_{\partial
S}=g|_{\partial S}$. A classical theorem
states that, the group $Map(S)$ is generated by Dehn twists along simple
closed curves.

Let $P$ be a unimodal polynomial of degree $d$, such that the Julia set $J(P)$
is homeomorphic to a Cantor set. This is equivalent to say that the critical
orbit of $P$ escapes to infinity. Let $\gamma$ be a Jordan curve, whose
interior contains the critical value and the Julia set. The preimage
$P^{-1}(\gamma)$ consists of $d$ disjoint Jordan curves and, together with
$\gamma$, defines a $d+1$-connected Riemann surface $S_1$ with boundary. Define
recursively $S_n$ by $S_n= P^{-1}(S_{n-1})$. Then $Map(S_{n+1})$ is the $d$-fold
product of $Map(S_n)$. We have the following:

\begin{lemma}\label{direct.lem} Let $P$ be a unimodal polynomial such that
$J(P)$ is homeomorphic to a Cantor set, then $Map(S_n)$ is embedded
into $Mod_2(P)$. Thus $\varinjlim Map(S_n)$ is also embedded into $Mod_2(P)$.
\end{lemma}
\begin{proof}
The embedding from $S_n$ to $S_{n+1}$, induces a monomorphism
from the
group $Map(S_n)$ to the group $Map(S_{n+1})$. To conclude the lemma,
we show that every element in $Map(S_n)$ induces a non-trivial element
in $Mod_2(P)$. Let $\tau$ be a Dehn twist along a simple closed curve
$\gamma$. Using dynamics of $P$, we propagate $\tau$ along the great orbit of
$\gamma$. This defines an element in $Mod_2(P)$. Thus, we have a map
$\Phi_n:Map(S_n) \rightarrow Mod_2(P)$. If $\tau\neq \tau'$ in $Map(S_n)$, then
$\tau$ and $\tau'$ have different rotation numbers along the same curves. But,
this property is preserved by the dynamics of $P$ and then
$\Phi_n(\tau)\neq\Phi_n(\tau')$. So $\Phi_n$ is an injective map.
\end{proof}
Note that if we consider, instead of $Map(S_n)$, the group of automorphisms
of $S_n$, not necessarily acting identically on $\partial S_n$. Then,
on the corresponding product, it appears the action of a braiding group.

It is not clear that every element in $Mod_2(R)$, acting identically on
$J(R)$, should be homotopic to some element in $Map(S_n)$. Moreover, $Mod_2(R)$
consists of elements that have a simplicial extension, this relates the modular
group $Mod_2(R)$ with Thompson's group of automorphisms of the Cantor set.

In general, $T_2(R)$ is not path connected. Since $T_2(R)$ contains $H(T_1(R))$
and the orbit of $H(T_1(R))$ under the action of $Mod_2(R)$. Locally, the orbit
space is homeomorphic to $Mod_2(R)/\alpha(Mod_1(R))$. Thus we have

\begin{lemma} If the homomorphism $\alpha$ is not surjective, the space
$$Mod_2(R)/\alpha(Mod_1(R))$$ is totally disconnected.
\end{lemma}

\begin{proof} Assume that there is a path $\sigma:[0,1]\rightarrow
Mod_2(R)/\alpha(Mod_1(R))$. But, the path $\sigma$ induces a homotopy between
$\sigma(0)$ and $\sigma(1)$, for all $t\in [0,1]$. Hence $\sigma$ is a constant
map in $Mod_2(R)$. \end{proof}

\begin{theorem} Let $P$ be a hyperbolic unimodal polynomial such that $J(P)$
is homeomorphic to a Cantor set. Then $$T_2(P)=H(T_1(P))\times
\{ Mod_2(P)/\alpha(Mod_1(P))\}.$$
Moreover, the space $$Mod_2(P)/\alpha(Mod_1(P))$$ is perfect.
\end{theorem}

\begin{proof}
Let us check first, that the space $Mod_2(P)/\alpha(Mod_1(P))$ is perfect. By
Lemma \ref{direct.lem}, $Mod_2(P)$ contains $\varinjlim Map(S_n)$. Let
$\gamma_1$ be simple closed curve in $S_n$. For
every $n>1$, choose a component $\gamma_n$ of $P^{-n}(\gamma)$. Let $g_n$ be
the map, in $Mod_2(P)$, induced by the Dehn twist of angle $2\pi$ along
$\gamma_n$, and acting as the identity around all other components of
$P^{-n}(\gamma)$.

Then, the maps $g_n$ are different, and by construction, $g_n$ can not be
extended to a globally defined element in $Mod_1(P)$ commuting with dynamics.
Moreover, the $g_n$ belong to different orbits of the action of
$\alpha(Mod_1(P))$. Thus, the maps $g_n$ induce different elements in
$Mod_2(P)/\alpha(Mod_1(P))$. Moreover, the distortions satisfy $K(g_n)=K(g_1)$,
for all $n$. Hence, the map $g_n(Id)$ belongs to the ball $B(Id,K(g_1)+1)$ in
$T_2(P)$.
This implies that, there exist an accumulation point in
$B(Id,K(g_1)+1)$, and thus there is an accumulation point in 
$Mod_2(P)/\alpha(Mod_1(P))$. Since, the action of the group $Mod_2(P)$ is
transitive in $Mod_2(P)/\alpha(Mod_1(P))$, it follows that the fiber
$Mod_2(P)/\alpha(Mod_1(P))$ is perfect.

By Lemma \ref{Dehn.lem}, the map $\alpha$ is injective and, then
$H(T_1(P))\simeq T_1(P)$ by Proposition \ref{ker.a}. Since
$Mod_1(P)$ acts properly discontinuously on $T_1(P)$, there exist
$r_0>0$ such that the ball $B(Id,r_0)$, in $T_1(P)$, projects injectively into
$QC_(P)$. Then $B(Id,r_0)$ embeds injectively on $Hyp(P)$, and the image of
$B(Id,r_0)$ in $Hyp(P)$
is evenly covered by the projection of $T_2(P)$ into $Hyp(P)$. Let
$U$ be the open component, in the fiber of $B(Id,r_0)$, containing the
identity in $T_2(P)$. By
construction, $U\cap Mod_2(P)/\alpha(Mod_1(P))=Id$.
Then, $Id$ has a neighborhood in $T_2(P)$ of the form $$U\times
Mod_2(P)/\alpha(Mod_1(P)).$$ The same argument works for every
$x\in T_2(P)$. Hence $T_2(P)$ is homeomorphic to the product $T_1(P)\times
Mod_2(P)/\alpha(Mod_1(P))$.
\end{proof}

Conceivably, $T_2(R)$ is also locally a product when $R$ is a
rational map with disconnected Julia set. However, in this case, a non-trivial
monodromy can appear along $H(T_1(R)).$

Now we will see that the connectivity of $J(R)$ is related to the injectivity of
$\alpha:Mod_1(R)\rightarrow Mod_2(R)$.

\begin{theorem} Let $P$ be a hyperbolic polynomial. The map
$\alpha:Mod_1(P)\rightarrow Mod_2(P)$ is injective if, and only if, $P$ is
unimodal and the Julia set $J(P)$ is homeomorphic to a Cantor set.
 \end{theorem}

\begin{proof}
Let us assume that $P$ is unimodal, and $J(P)$ is homeomorphic to a Cantor set.
By Lemma \ref{Dehn.lem}, the modular group $Mod_1(P)$ is cyclically generated
by a Dehn twist $\tau$, but $\tau$ is non-trivial in $Mod_2(P)$. Hence,
$\alpha$ is injective.

Reciprocally, by the hyperbolicity of $P$, the critical point in $\C$ is
attracted to a periodic cycle in the plane. Without loss of generality, we can
assume that this periodic cycle is a fixed point $z_0$. The immediate basin of
attraction $A(z_0)$ is a topological disk. Let $A$ be an annulus inside
$A(z_0)$, with center at $z_0$, such that $P$ maps the outer boundary of $A$ to
the inner boundary of $A$. Consider a Dehn twist along the core curve of $A$ and
propagate it along its grand orbit using dynamics. The resulting  map $\tau$ is
a non trivial element of $Mod_1(P)$. However, as we saw in Example
\ref{Blashke.sur}, near the boundary of $A(z_0)$, $\tau$ is
homotopic to the identity. Thus, $\tau$ is the identity in $Mod_2(P)$ and
$\alpha$ is not injective in $Mod_1(P)$. Then, the critical point is
attracted to infinity. Hence, the Julia set $J(P)$ is homeomorphic to a
Cantor set. \end{proof}

\begin{Example}

A useful example is $f_{10}(z)=z^2+10$. In this case, the Julia set is a
Cantor set. So the group $Mod_1(f_{10})$ is cyclically generated by a Dehn
twist. It follows that $T_1(f_{10})$ is homeomorphic to the puncture unit disk.
The quotient space is equivalent to the complement of the Mandelbrot set $M$. It
is well known, that $\C\setminus M$ is holomorphically equivalent to the
puncture unit disk. Also, since $Mod_2(f_{10})$ is infinitely generated, the
homomorphism $\alpha$ is not surjective.

\end{Example}

\section{Inverse limits of rational functions and its deformations.}

Let us consider a rational map $R:\bar{\C}\rightarrow \bar{\C}$ acting on the
Riemann sphere. The inverse limit, or natural extension of $R$, is the space
$$\mc{N}_R=\{\h{z}=(z_1,z_2,...)\in \prod_{n\in \mathbb{N}} \bar{\C}:
R(z_{n+1})=z_n \}$$
endowed by Tychonoff topology as a subspace of $\prod_{n\in \mathbb{N}}
\bar{\C}$. There is a family of natural projections $p_n:\mc{N}_R\rightarrow
\bar{\C}$ given by $p_n(\h{z})=z_n$, also a natural extension of $R$, denoted
by $\h{R}:\mc{N}_R\rightarrow \mc{N}_R$ such that $p_n\circ \h{R}=R\circ p_n$.
To simplify notation, let us put $p:=p_1.$

The space $\mc{N}_R$ was studied by Lyubich and Minsky in \cite{LM}. In that
paper, Lyubich and Minsky showed that for a general rational map, the natural
extension is decomposed into two spaces; the \textit{regular part} $\mc{R}_R$,
which consist of the points that admit a Riemannian structure compatible with
the maps $p_n$, and the complement of $\mc{R}_R$ called the \textit{irregular
part}. A \textit{leaf} $L$ is a path-connected component in $\mc{R}_R$. Every
leaf is a Riemann surface. A theorem by Lyubich and Minsky shows that, in
$\mc{R}_R$, there is a family of leaves, such that, each leaf in this family is
conformally equivalent to the plane and it is dense in $\mc{N}_R$.

The authors of \cite{LM}, proved that there is a class of rational maps $R$,
which contains all hyperbolic maps, such that $\mc{R}_R$ is a lamination by
Riemann surfaces. That is, it admits an atlas of charts $(U,\phi)$, where $\phi$
is a homeomorphism from $U$ to $\mathbb{D}\times T$. Changes of coordinates are
conformal on the horizontal direction, and continuous in the transversal
direction. This structure is consistent with the fibration induced by the family
of maps $p_n$.

Let $P(R)$ denote the \textit{postcritical set} of $R$. If $z_0$ is a given
point in $\C\setminus P(R)$ then, a construction due to Poincar\'e gives
a representation of the fundamental group $\pi_1(\C\setminus P(R),z_0)$, into
the automorphisms group of the fiber $p^{-1}(z_0)$. The image of this
representation is called
the \textit{monodromy group} of $\mc{N}_R$.  Because of the irregular part, the
natural extension is not the suspension of $\C$ by the monodromy group on the
fiber $p^{-1}(z_0)$.

\subsection{Deformations of inverse limits.}

Consider an open neighborhood $U$ of the Julia set $J(R)$, we call
the fiber $p^{-1}(U)$ a \textit{maximal flow box} for $\mc{N}_R$.
The action of the monodromy group induces identifications on a
maximal flow box. We say that $\mc{N}_R$ can be represented by a
maximal flow box and the action of monodromy if, $\mc{N}_R$ coincides
with the end compactification of the orbit, of a maximal flow box, by
the action of monodromy. It is not clear if, in general, the natural
extension of a rational map $R$ can be represented by a maximal flow
box. However, this is true when $R$ is hyperbolic. In this case, the
regular part is a Riemann surface lamination and the irregular part
is finite, see \cite{LM}.

From now on, we assume that $\mc{N}_R$ is represented by a maximal
flow box. Then any conjugacy, around a neighborhood of the Julia set
of $R$, can be extended to a homeomorphism of the whole laminations.
This suggests that, we can extend the equivalence class of elements
in $T_2(R)$ to equivalence classes of laminations. In this sense,
the monodromy and the dynamics characterize laminations. Then,
deformations of the whole lamination are determined by deformations
of a maximal flow box.

Let $U$ be a neighborhood in $\mc{N}_R$, we call a \textit{plaque} a
path component of $U\cap \mc{N}_R$. A map $\gamma$, continuously
defined on plaques or open neighborhoods in $\mc{N}_R$, is called
a \textit{fiber automorphism} if $p\circ \gamma=p$. Since $p$ is
holomorphic in the regular part, it implies that $\gamma$ is
holomorphic in $\mc{R}_R$. Given a fiber automorphism and a leaf
$L$ in $\mc{R}_R$, we denote by $\gamma_L$ the restriction of
$\gamma$ to $L$ when is defined.

\begin{definition} Let $U$ be a neighborhood in $\C$ and $F=p^{-1}(U)$ a
flow box. A family $\{\mu_L\}$ of Beltrami differentials, defined on
$F$, is called compatible with the fiber structure if, for every
fiber automorphism $\gamma$ and every leaf $L$ in $\mc{R}_R$, we
have on $F\cap \gamma(F)\cap
L$,$$\mu_{\tilde{L}}(\gamma_L)\frac{\bar{
\gamma'_L}}{\gamma'_L}=\mu_{\tilde{L}},$$ where $\gamma_L$ sends $L$
into $\tilde{L}$.
\end{definition}

Let $\{R^{-n}\}$ be the family of branches of $R$, then deck
transformations of the family of branches are fiber automorphisms.
Moreover, all fiber automorphisms are generated by deck
transformations of branches of $R$.

\begin{lemma}\label{lem.comp}
Let $\mu=\{\mu_L\}$ be a family of Beltrami differentials in
$\mc{R}_R$, then $\{\mu_L\}$ is compatible with the fiber structure
if, and only if, $p_*\circ p^*(\mu)=\mu$.
\end{lemma}

\begin{proof}
Assume that $\mu$ is compatible with the fiber structure, then it is
invariant under all deck transformation of branches of $R$. Thus the
push forward $p^*(\mu_L)$ is independent of the leaf $L$ and, the
pull-back $p_*\circ p^*(L)$ is the same for all leaves $L$ and coincides
with $\mu$.

The equation $p_*\circ p^*(\mu)=\mu$ implies that, the family $\mu$
is invariant under deck transformations of $R$. Hence $\mu$ must be
compatible with the fiber structure.
\end{proof}

Since the natural extension $\mc{N}_R$ is a metric space, we will consider
quasiconformal maps, in Pesin's sense, defined on subsets of $\mc{N}_R$. Let $F$
be a maximal
flow box for $\mc{N}_R$, let  $X(\mc{N}_R,F)$ be the space of
surjective homeomorphisms $\phi:\mc{N}_R\rightarrow \mc{N}_{R_1}$,
quasiconformal in Pesin's sense, such that on $F$ conjugates the
monodromy actions. This condition implies that $\phi$ induces a
family of Beltrami differentials on $F$, compatible with the fiber
structure.

We say that two maps $\phi:\mc{N}_R\rightarrow \mc{N}_{R_1}$ and
$\psi:\mc{N}_R\rightarrow \mc{N}_{R_2}$, in $X(\mc{N}_R,F)$, are
equivalent if there exist a map $\sigma:\mc{N}_{R_1}\rightarrow
\mc{N}_{R_2}$, conformal in Pesin's sense, such that
$\psi=\sigma\circ \phi$ and $\sigma$ is homotopic to the identity,
with homotopy that commutes with dynamics and monodromy actions.

We define the space of deformations of $\mc{N}_R$, and denote it by
$Def(\mc{N}_R,F)$, as the set $X(\mc{N}_R,F)$ modulo the equivalence
relation above. Since we are considering surjective homeomorphisms
$\phi:\mc{N}_R\rightarrow \mc{N}_{R_1}$, it follows that $R$ and
$R_1$ have the same degree.

Note that the image of $F$ under any map in $Def(\mc{N}_R,F)$ is a
maximal flow box of some rational map. We have

\begin{theorem} Let $R$ be a map that is represented by a maximal flow box.
Then, there is a bijection between $Def(\mc{N}_R,F)$ and the space
$T_2(R).$
\end{theorem}

\begin{proof}
Let $(h,U)$ be an element in $T_2(R)$, then $h$ induces a Beltrami
differential $\nu$ around a neighborhood $U$ of $J(R)$, we consider the
family of Beltrami differentials $\mu$ on a neighborhood of
$p^{-1}(J(R))$. Since $R$ is represented by a maximal flow box, we
can use dynamics and monodromy to propagate $\mu$ on all the inverse
limit. By construction $p_*\circ p^*(\mu)=\mu$, so by Lemma
\ref{lem.comp}, the resulting family of Beltrami differentials is
compatible with the fiber structure. Thus $\mu$ induces an element
in $Def(\mc{N}_R,F)$ and, the construction only depends on
the class of $(h,U)$ in $T_2(R)$.

Now, let $\phi$ be a representative of a point in $Def(\mc{N}_R,F)$.
Since $\phi$ conjugates dynamics and the monodromy actions, we can
push $\phi$, using $p$, to get a quasiconformal map $h:U\rightarrow
\C$, conjugating the corresponding rational maps. Then, $(h,U)$
defines an element in $T_2(R)$.
\end{proof}

\bibliographystyle{amsplain}
\bibliography{workbib}

\end{document}